\documentclass[12pt, reqno]{amsart}
\usepackage{lineno,hyperref}

\usepackage{cases}
\usepackage[square, comma, sort&compress, numbers]{natbib}
\usepackage{float}
\usepackage{pifont}
\usepackage{bbding}
\usepackage{amssymb}
\usepackage{mathrsfs}
\usepackage{subfigure}
\usepackage{caption}
\usepackage{geometry}
\usepackage{amsmath, amsthm, amscd, amsfonts, amssymb, graphicx, color}
\usepackage{lineno}
\newtheorem{thm}{Theorem}
\newtheorem{theorem}{Theorem}[section]
\newtheorem{lemma}[theorem]{Lemma}
\newtheorem{proposition}[theorem]{Proposition}
\newtheorem{corollary}[theorem]{Corollary}
\theoremstyle{definition}
\newtheorem{definition}{Definition}
\newtheorem{example}{Example}

\newtheorem{remark}{Remark}
\newtheorem{conj}{Conjecture}
\newtheorem{algor}{Algorithmic}
\newtheorem{case}{Case}
\theoremstyle{remark}
\numberwithin{equation}{section}
\DeclareMathOperator{\RE}{Re}
\DeclareMathOperator{\IM}{Im}

\newcommand{\ID}{\mathbb{D}}

\newcommand{\ds}{\displaystyle}
\begin{document}
\thispagestyle{empty} \setcounter{page}{1}


\title[On univalent log-harmonic mappings ]
{On univalent log-harmonic mappings}
\author[Z. Liu]{ZhiHong Liu }
\address{Z. Liu, 
College of Science, Guilin University of Technology, Guilin 541006, Guangxi, People's Republic of China.}
\email{\textcolor[rgb]{0.00,0.00,0.84}{liuzhihongmath@163.com}}

\author[S. Ponnusamy]{Saminathan Ponnusamy 
}
\address{S. Ponnusamy, Department of Mathematics,
Indian Institute of Technology Madras, Chennai-600 036, India.
}
\email{\textcolor[rgb]{0.00,0.00,0.84}{samy@iitm.ac.in}}




\begin{abstract}
We consider the class univalent log-harmonic mappings  on the unit disk.
Firstly, we obtain necessary and sufficient conditions for a complex-valued continuous function to be starlike or convex
in the unit disk. Then we present a general idea, for example, to construct log-harmonic Koebe mapping, log-harmonic
right half-plane mapping and log-harmonic two-slits mapping and then we show precise ranges of these mappings. Moreover,
coefficient estimates for univalent log-harmonic starlike mappings are obtained. Growth and distortion theorems for
certain special subclass of log-harmonic mappings are studied. Finally, we propose two conjectures, namely, log-harmonic
coefficient and log-harmonic covering conjectures.
\end{abstract}
\maketitle

\section{Introduction and preliminary results}
\noindent Let $\mathcal{A}$ be the linear space of all analytic functions defined on the open unit disc $\ID$, and
let $\mathcal{B}$ be the set of all functions $\mu\in \mathcal{A}$ such that $|\mu(z)| < 1$ for all $z\in \mathbb{D}$.
A log-harmonic mapping is a solution of the nonlinear elliptic partial differential equation
\begin{equation}\label{dfLH}
\overline{f_{\overline{z}}(z)}=\mu(z)\left(\frac{\overline{f(z)}}{f(z)}\right)f_{z}(z),
\end{equation}
where $\mu$ is the second complex dilatation of $f$ and $\mu\in \mathcal{B}$. Thus, the Jacobian $J_{f}$ of $f$ is given by
$$J_{f}=|f_z|^2-|f_{\overline{z}}|^2=|f_z|^2(1-|\mu|^2),
$$
which is positive, and therefore, every non-constant log-harmonic mapping is sense-preserving and open in $\ID$. If $f$ is a
non-vanishing log-harmonic mapping in $\ID$, then $f$ can be expressed as
\begin{equation*}
f(z)=h(z)\overline{g(z)},
\end{equation*}
where $h(z)$ and $g(z)$ are non-vanishing analytic functions in $\ID$.
On the other hand, if $f$ vanishes at $z=0$ but is not identically zero, then such a mapping $f$ admits the following representation
\begin{equation*}
f(z)=z|z|^{2\beta}h(z)\overline{g(z)},
\end{equation*}
where $\RE \beta>-1/2$, $h,g\in\mathcal{A}$,  $h(0)\neq 0$ and $g(0)=1$ (cf.~\cite{Abdulhadi1988}).
The class of functions of this form has been widely studied. See, for example~\cite{Ali2016,Mao2013,AbdulHadi2015}.

For simplicity, we set $\beta=0$ and consider the class $\mathcal{S}_{Lh}$ of univalent log-harmonic mappings $f$ of the form
\begin{equation*}
f(z)=zh(z)\overline{g(z)},
\end{equation*}
where $h,g\in\mathcal{A}$ with the normalization $h(0)=g(0)=1$, and such that
\begin{equation}\label{LHBJC}
h(z)=\exp\left(\sum_{n=1}^{\infty}a_{n}z^n\right)\quad{\text{and}}\quad g(z)=\exp\left(\sum_{n=1}^{\infty}b_{n}z^n\right).
\end{equation}

A complex-valued function $f:\,\Omega\to \mathbb{C}$ is said to belong to the class $C^{1}(\Omega)$ (resp. $C^{2}(\Omega)$)
if $\RE f$ and $\IM f$ have continuous first order (resp.~second order) partial derivatives in $\Omega$. For $f\in C^{1}(\Omega)$,
consider the complex linear differential operator $Df$ defined on $C^{1}(\Omega)$ by
$$Df=zf_{z}-\overline{z}f_{\overline{z}}.
$$
In view of the Riemann mapping theorem, it suffices to consider the case when $\Omega=\ID$.
\begin{definition}
Let $\alpha\in[0,1)$. A univalent function $f\in C^{1}(\ID)$ with $f(0)=0$ is called \emph{starlike of order $\alpha$},
denoted by $\mathcal{FS}^{*}(\alpha)$, if
\begin{equation*}
\frac{\partial}{\partial\theta} \left(\arg f(re^{i\theta})\right)
=\RE\left(\frac{Df(z)}{f(z)}\right)
=\RE\left(\frac{zf_z(z)-\overline{z}f_{\overline{z}}(z)}{f(z)}\right)>\alpha,
\end{equation*}
for all $z=re^{i\theta}\in\ID\backslash\{0\}$. We set $\mathcal{FS}^{*}(0)=:\mathcal{FS}^{*}$, and functions in
$\mathcal{FS}^{*}$ are $C^{1}$-(fully) starlike in $\ID$. If  $f\in C^{1}(\ID)$
is replace by $f\in {\mathcal S}$,  then $\mathcal{FS}^{*}(\alpha)$ coincides with the class $\mathcal{S}^{*}(\alpha)$,
of analytic starlike functions of order $\alpha$.
\end{definition}

We denote by $\mathcal{S}_{H}^{*}(\alpha)$ and $\mathcal{S}^{*}_{Lh}(\alpha)$ the set of all starlike harmonic functions of
order $\alpha$ and starlike log-harmonic functions of order $\alpha$, respectively. If $\alpha=0$, then we simply denote these classes
by $\mathcal{S}_{H}^{*}$ and $\mathcal{S}^{*}_{Lh}$, respectively. Thus, they are called the set of all starlike harmonic functions
and starlike log-harmonic functions, respectively. These classes are investigated in detail by a number of authors. See for example \cite{Abdulhadi2006,AbdulHadi2015}.

The following theorem establishes a link between the classes $\mathcal{S}^{*}_{Lh}(\alpha)$ and $\mathcal{S}^{*}(\alpha)$.

\begin{thm}\label{thmLSS}
{\rm(\cite[Lemma 2.4]{Abdulhadi1987} and~\cite[Theorem 2.1]{Abdulhadi2006})}
Let $f(z)=zh(z)\overline{g(z)}$ be a log-harmonic mapping on $\ID$, $0\not\in (hg)(\ID)$. Then $f\in\mathcal{S}^{*}_{Lh}(\alpha)$
if and only if $\varphi(z)=zh(z)/g(z)\in\mathcal{S}^{*}(\alpha)$.
\end{thm}

In~\cite{LiPei2013},  Li,  et al. proved the following result.
\begin{thm}\label{thmLSW}
Let $f(z)=\varphi(z)|z|^{2(p-1)}~(p\geq 1)$, where $\varphi\in C^{1}(\ID)$ is starlike (not necessarily harmonic) in $\ID$.
Then $f\in C^{1}(\ID)$ is starlike and univalent in $\ID$.
\end{thm}

We state the following simple result which generalizes both Theorems \ref{thmLSS} and \ref{thmLSW}.
\begin{proposition}\label{thmGSC}
Let $f(z)=\varphi(z)|g(z)|^2$ be a complex-valued function on $\ID$, where $\varphi, g\in \mathcal{A}$ and $\varphi(z),~g(z)\neq 0$
in $\ID\backslash\{0\}$. Then $f\in \mathcal{FS}^{*}(\alpha)$ if and only if $\varphi\in \mathcal{FS}^{*}(\alpha)$.
\end{proposition}
\begin{proof}
A simply calculation shows that
\begin{equation*}
zf_{z}(z)=z \varphi_{z}(z)|g(z)|^2+\varphi(z) z g'(z)\overline{g(z)}=\left(\frac{z\varphi_{z}(z)}{\varphi(z)}+\frac{zg'(z)}{g(z)}\right)f(z)
\end{equation*}
and similarly
\begin{equation*}
\overline{z}f_{\overline{z}}(z)=\left(\frac{\overline{z}\varphi_{\overline{z}}(z)}{\varphi(z)}
+\overline{\left(\frac{zg'(z)}{g(z)}\right)}\right)f(z),
\end{equation*}
which clearly implies that
\begin{equation*}
\begin{split}
\RE\left(\frac{zf_z(z)-\overline{z}f_{\overline{z}}(z)}{f(z)}\right)
=\RE\left(\frac{z\varphi_z(z)-\overline{z}\varphi_{\overline{z}}(z)}{\varphi(z)}\right).
\end{split}
\end{equation*}
The desired conclusion follows.
\end{proof}
\begin{definition}
Let $\alpha\in[0,1)$. A function $f\in C^{2}(\ID)$ with $f(0)=0$ and $\frac{\partial }{\partial\theta}f(re^{i\theta})\neq 0$,
$0<r<1$, is called a {\em fully convex of order $\alpha$}, denoted by $\mathcal{FC}(\alpha)$, if
\begin{equation*}
\begin{split}
\frac{\partial}{\partial\theta}\left(\arg\frac{\partial}{\partial\theta}f(re^{i\theta})\right)
&=\RE\left(\frac{D^2 f(z)}{Df(z)}\right)\\
&=\RE\left(\frac{zf_{z}(z)+\overline{z}f_{\overline{z}}(z)-2|z|^2f_{z\overline{z}}(z)+z^2f_{zz}(z)
+\overline{z}^2f_{\overline{z}\,\overline{z}}(z)}{zf_{z}(z)
-\overline{z}f_{\overline{z}}(z)}\right)>\alpha,
\end{split}
\end{equation*}
for $z=re^{i\theta}\in\ID\backslash\{0\}$ (see also~\cite{Chuaqui2004,LiPei2013} in order to distinguish convexity in the
analytic and the harmonic cases), where $D^2 f=z(Df)_z-\overline{z}(Df)_{\overline{z}}$. Set $\mathcal{FC}(0)=:\mathcal{FC}$,
the class of {\em fully convex} (univalent) on $\ID$. In the analytic case, $\mathcal{FC}(\alpha)$ coincides with the class
$\mathcal{C}(\alpha)$ of convex functions of order $\alpha$.
\end{definition}

In particular, we denote by $\mathcal{FC}_{H}^{0}(\alpha)$ and $\mathcal{FC}_{Lh}(\alpha)$ the set of all fully convex harmonic
functions of order $\alpha$ and fully  convex log-harmonic functions of order $\alpha$, respectively. If $\alpha=0$, we denote
these classes by $\mathcal{FC}_{H}^{0}$ and $\mathcal{FC}_{Lh}$, the set of all fully convex harmonic functions and fully
convex log-harmonic functions, respectively. Throughout the paper, we treat fully starlike (or convex) mappings as starlike
(or convex) mappings although this is not the case in strict sense.

We remark that the function $f(z)=\varphi(z)|g(z)|^2$ is not necessarily convex in $\ID$ even if $\varphi(z)$ is a complex-valued
convex mapping in $\ID$. For example, consider
\begin{equation*}
\varphi(z)=\frac{z}{1-z},\quad g(z)
=\exp\left(\frac{1}{4}\log\left(\frac{1-z}{1+z}\right)\right)\exp\left(\frac{z}{2(1-z)}\right)
\end{equation*}
and therefore, by \eqref{dfLH}, $f$ defined by
$$f(z)=\varphi(z)|g(z)|^2=\frac{z}{1-z}\left|\frac{1-z}{1+z}\right|^{1/2}
\exp\left(\RE\left(\frac{z}{1-z}\right)\right)
$$
is a univalent log-harmonic mapping in $\ID$ but is not convex in $\ID$ although $\varphi(z)=z/(1-z)$ is convex in $\ID$.
The image of $\ID$ under $f(z)$ is shown in Figure~\ref{fLHf}.

For an analytic proof of this fact, we need to show that
$$V(r,\theta)=:\RE\left\{\frac{D^2f\left(re^{i\theta}\right)}
{Df\left(re^{i\theta}\right)}\right\}<0
$$
for some $0<r<1$ and $0\leq\theta< 2\pi$. Since $f=\varphi |g|^2$, by direct calculation, we find that
\begin{equation*}
\begin{split}
  Df&=z\varphi'g\overline{g}+z\varphi g'\overline{g}-\varphi g\overline{z}\overline{g'}
\end{split}
\end{equation*}
and
\begin{equation*}
\begin{split}
  D^2f&=z\varphi'g\overline{g}+z^2\varphi''g\overline{g}+2z^2\varphi'g'\overline{g}
  +z\varphi g'\overline{g}+z^2\varphi g''\overline{g}-2z\varphi'g\overline{z}\,\overline{g}\\
  &\quad-2z\varphi g'\overline{z}\overline{g'}+\varphi g\overline{z}^2\overline{g''}.
\end{split}
\end{equation*}

Using this, we see that $V(r,\theta)<0$ for some values of $r$ and $\theta$. To avoid the technical details, we compute
$V(r,\theta)$ for $r=8/9$, and $\theta=\pi/3,~\pi/4,~2\pi/3$ with the help of \emph{Mathematica} (see Table \ref{tab1}).
According to this, we obtain that $f(z)$ is not convex in $\ID$.
\begin{table}[h]
\centering
\caption{$V(r,\theta)$ for values of $r$ and $\theta$}\label{tab1}
\begin{tabular}{|c|c|c|}
\hline
$r$ & $\theta$ & $V(r,\theta)$ \\\hline
8/9 & $\pi/4$ & -0.284821\\
\hline
8/9 & $\pi/3$ & -0.447807\\
\hline
8/9 & $2\pi/3$ & -0.510244\\
\hline
\end{tabular}
\end{table}

\begin{figure}[!h]
  \centering
  \includegraphics[height=3.0in,keepaspectratio]{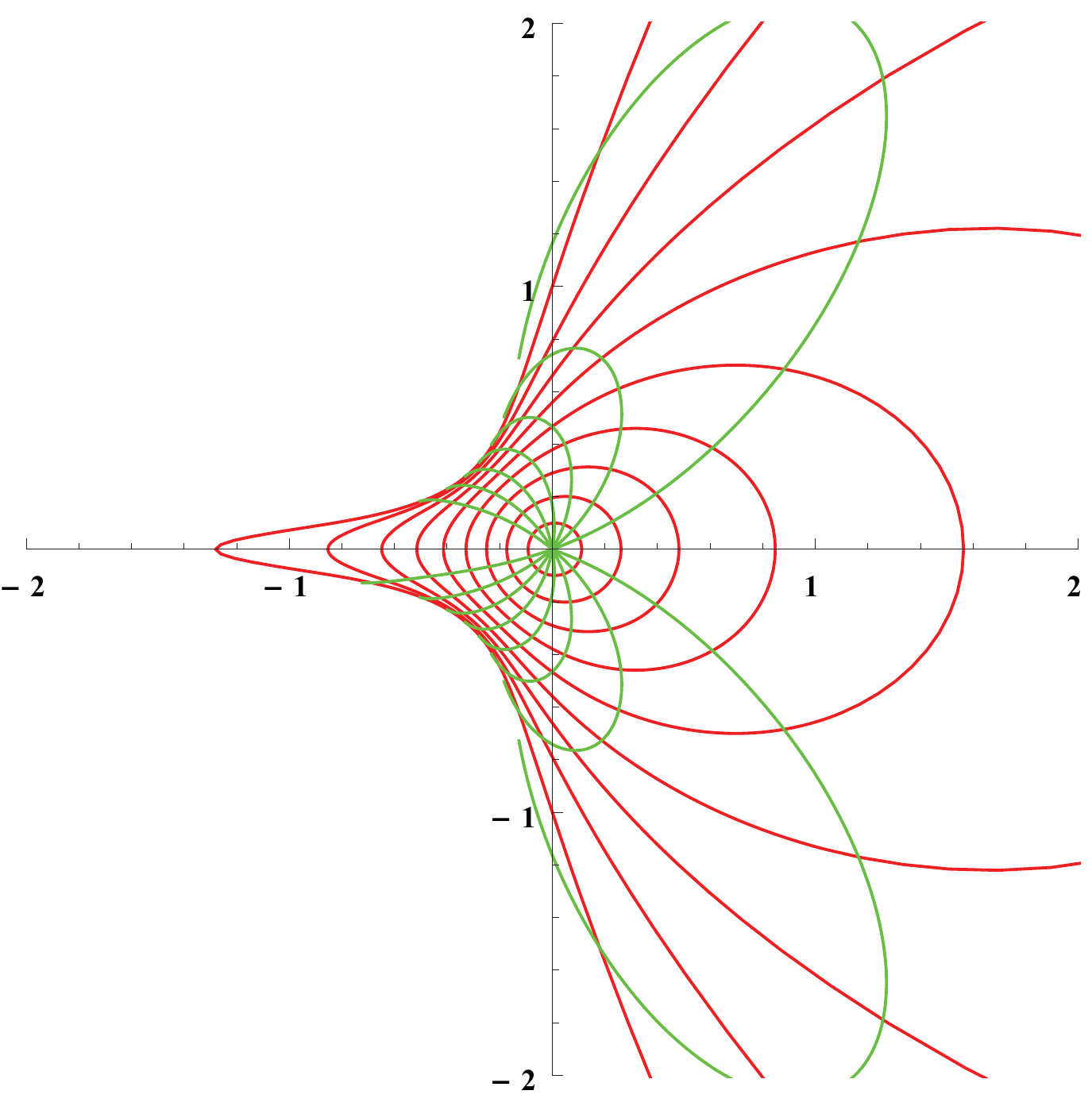}
\caption{Image of $\ID$ under $f(z)$}\label{fLHf}
\end{figure}

However, with $g(z)=z^{p-1}~(p\geq 1)$, we can obtain a necessary and sufficient condition for  the function $f$ of the form
$f(z)=\varphi(z)|z|^{2(p-1)}$ to be convex.
\begin{theorem}\label{thmNSCC}
Let $f(z)=\varphi(z)|z|^{2(p-1)}~(p\geq 1)$. Then $f\in \mathcal{FC}(\alpha)$ if and only if $\varphi\in \mathcal{FC}(\alpha)$.
\end{theorem}
\begin{proof}
We compute
\begin{equation*}
zf_{z}(z)=|z|^{2(p-1)}\left[z\varphi_{z}(z)+(p-1)\varphi(z)\right]
\end{equation*}
and similarly,
\begin{equation*}
\overline{z}f_{\overline{z}}(z)=|z|^{2(p-1)}
\left[\overline{z}\varphi_{\overline{z}}(z)+(p-1)\varphi(z)\right].
\end{equation*}
Thus, we have
\begin{equation*}
 Df=zf_{z}-\overline{z}f_{\overline{z}}
=|z|^{2(p-1)}(z\varphi_{z}-\overline{z}\varphi_{\overline{z}}),
\end{equation*}
which may be conveniently written as $F(z)=|z|^{2(p-1)}\Phi(z)$. Thus, to compute $D^{2}f=D(Df)=DF$, we first find that
\begin{equation*}
z\Phi_{z}(z)-\overline{z}\Phi_{\overline{z}}(z)
=z\left(\varphi_{z}(z)+z\varphi_{zz}(z)-\overline{z}\varphi_{\overline{z}z}(z)\right)-
\overline{z}\left(z\varphi_{z\overline{z}}(z)-\varphi_{z}(z)
-\overline{z}\varphi_{\overline{z}\,\overline{z}}(z)\right),
\end{equation*}
so that
\begin{equation*}
\begin{split}
D^{2}f=|z|^{2(p-1)}\left(z\Phi_{z}-\overline{z}\Phi_{\overline{z}}\right)
=|z|^{2(p-1)}\left(z\varphi_{z}+\overline{z}\varphi_{\overline{z}}-2|z|^2\varphi_{z\overline{z}}
+z^2\varphi_{zz}+\overline{z}^2\varphi_{\overline{z}\,\overline{z}}\right).
\end{split}
\end{equation*}
Finally, we see that
$$\RE\left(\frac{D^2f}{Df}\right)=\RE\left(\frac{z\varphi_{z}+\overline{z}\varphi_{\overline{z}}
-2|z|^2\varphi_{z\overline{z}}+z^2\varphi_{zz}+\overline{z}^2\varphi_{\overline{z}\,\overline{z}}}{z\varphi_{z}
-\overline{z}\varphi_{\overline{z}}}\right)=\RE\left(\frac{D^2\varphi}{D\varphi}\right),
$$
and this completes the proof.
\end{proof}

\begin{example}
Set $\varphi(z)=z-\lambda |z|^2$, where $0<|\lambda|<1/2$. It is easy to see that $\varphi(z)$ is log-harmonic mapping in $\ID$,
as a solution of~\eqref{dfLH} with the dilatation as
$$\mu(z)=\frac{\overline{\lambda} z}{1+\overline{\lambda} z}
$$
which is analytic in $\ID$ and $|\mu(z)|<1$ in $\ID$ (as $0<|\lambda|<1/2$). Simple calculation shows that
$$D\varphi(z)=z=D^{2}\varphi(z)
$$
and, for $0<|\lambda|<1/2$,
$$J_{\varphi}(z)=|1-\lambda \overline{z}|^2-|\lambda z|^2=1-2\RE(\lambda \overline{z})\geq 1-2|\lambda|>0.
$$
Consequently, $\varphi$ is convex and sense preserving in $\ID$, and by Theorem~\ref{thmNSCC}, we find that $f(z)=\varphi(z)|z|^{2(p-1)}$
is $\log$-$p$-harmonic convex in $\ID$. For more details about $\log$-$p$-harmonic mappings, see~\cite{LiPei2013} or \cite{LiPei2012}.
The images of $\ID$ under $\varphi(z)$ and $f(z)=\varphi(z)|z|^{2(p-1)}$ for $\lambda=1/4$ and $p=2$ are shown in Figure~\ref{f1}.
\end{example}

\begin{figure}[!h]
  \subfigure[Graph of $\varphi(z)=z-\frac{1}{4}|z|^2$]
  {\begin{minipage}[b]{0.45\textwidth}
  \includegraphics[height=2.6in,width=2.6in,keepaspectratio]{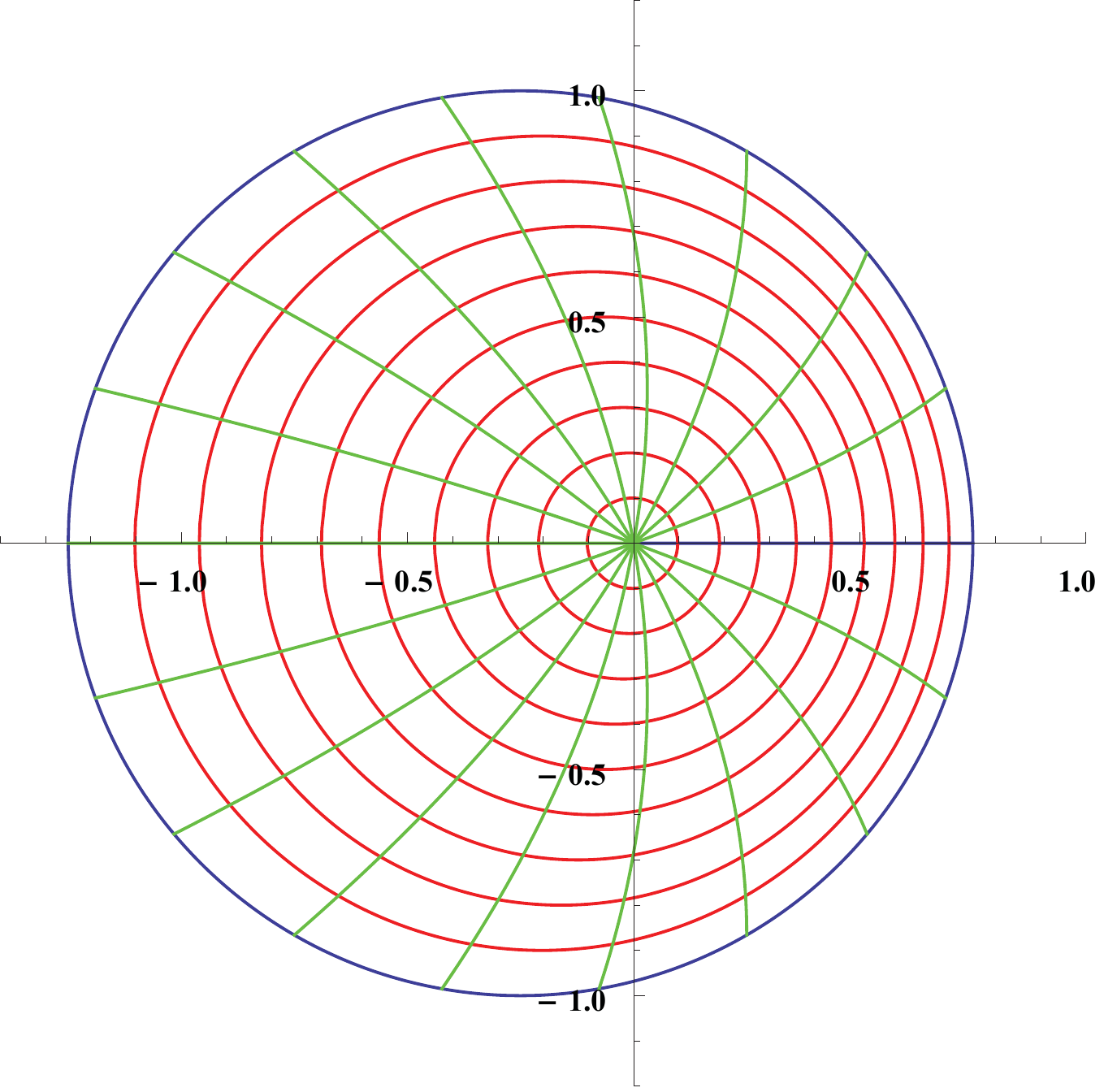}
\end{minipage}}
\subfigure[Graph of $f(z)=|z|^{2}(z-\frac{1}{4}|z|^2)$]
        {\begin{minipage}[b]{0.45\textwidth}
  \includegraphics[height=2.6in,width=2.6in,keepaspectratio]{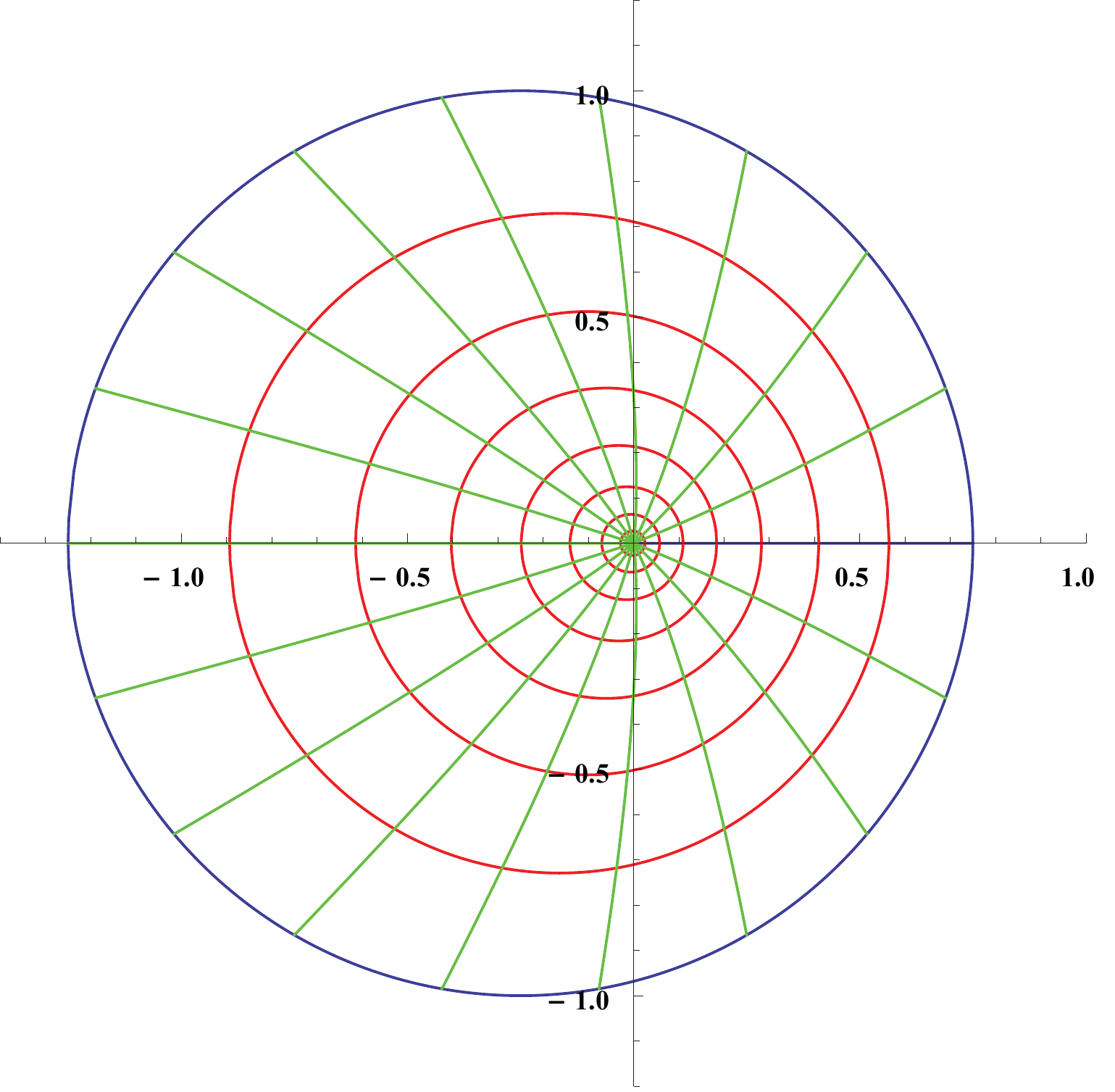}
\end{minipage}}
\caption{Images of $\ID$ under log-harmonic mapping $\varphi(z)=z-\lambda|z|^2$ and log-$p$-harmonic mapping
$f_{1}(z)=\varphi(z)|z|^{2(p-1)}$ for $\lambda=1/4$ and $p=2$.}\label{f1}
\end{figure}

The rest of the paper is arranged as follows: In Section~\ref{Sec2}, we construct some interesting univalent log-harmonic
mappings and show ranges of these mappings. In Section~\ref{Sec3}, we obtain sharp coefficient estimates for log-harmonic starlike
mappings. In Section~\ref{Sec4}, we study the growth and distortion theorems for the special subclass of $\mathcal{S}_{Lh}$.
Finally, in Section~~\ref{Sec5}, we propose coefficient conjecture and covering theorem for log-harmonic univalent mappings,
as an analog of Bieberbach conjecture.
\section{Construction of univalent log-harmonic mappings}\label{Sec2}
Clunie and Sheil-Small~\cite{Clunie1984} established a method of constructing a harmonic mapping onto a domain convex in
one direction by \emph{``shearing''} a given conformal mapping convex in the same direction. In ~\cite{Abdulhadi2012}, a
method was introduced for constructing univalent log-harmonic mappings $f(z)=zh(z)\overline{g(z)}\in \mathcal{S}_{Lh}$ from
the unit disk onto a strictly starlike domain $\Omega$.

Following the method of shearing-construction by Clunie and Shell-Small, we now present a method of log-harmonic mappings
with prescribed dilatation $\mu(z)$ with $\mu(0)=0$.
\begin{algor}\label{algor1}
Let $f(z)=zh(z)\overline{g(z)}$ be a sense-preserving log-harmonic mapping, where $h(z)$ and $g(z)$ are non-vanishing analytic
functions defined in $\ID$, normalized by $h(0)=g(0)=1$. Then the dilatation
$$\mu(z)=\frac{\overline{f_{\overline{z}}(z)}}{f_z(z)}\cdot\frac{f(z)}{\overline{f(z)}}
=\frac{zg'(z)/g(z)}{1+zh'(z)/h(z)}
$$
is analytic with $|\mu(z)|<1$ in $\ID$. According to the definition of log-harmonic mapping, the construction of log-harmonic
mappings proceeds by letting $zh(z)/g(z)=\varphi(z)$, where $\varphi$ is analytic satisfying $\varphi(0)=\varphi'(0)-1=0$, and
$\varphi(z)\neq 0$ for all $z\in \ID\backslash \{0\}$. This gives the pair of nonlinear differential equations
\begin{equation*}
\frac{zh(z)}{g(z)}=\varphi(z)\quad{\text{and}}\quad \frac{zg'(z)/g(z)}{1+zh'(z)/h(z)}=\mu(z),
\end{equation*}
which may be equivalently written as
\begin{equation*}
z\left(\log h\right)'(z)-z\left(\log g\right)'(z)=\frac{z\varphi'(z)}{\varphi(z)}-1\quad{\text{and}}\quad
z\left(\log g\right)'(z) =\mu(z)\left(1+z\left(\log h\right)'(z)\right).
\end{equation*}
Solving these two equations yield
$$\left(\log g\right)'(z) =\frac{\mu(z)}{1-\mu(z)}\cdot\frac{\varphi'(z)}{\varphi(z)}.
$$
Integrating with the normalization $g(0)=1$, we arrive at
\begin{equation}\label{EPRLHG}
g(z)=\exp\left(\int_{0}^{z}\left(\frac{\mu(s)}{1-\mu(s)}
\cdot\frac{\varphi'(s)}{\varphi(s)}\right)\,ds\right).
\end{equation}
In this way, we obtain the log-harmonic mapping $f$ defined by
\begin{equation}\label{EPRLHF}
f(z)=zh(z)\overline{g(z)}=\frac{zh(z)}{g(z)}|g(z)|^2
=\varphi(z)\exp\left(2\RE\int_{0}^{z}\left(\frac{\mu(s)}{1-\mu(s)}
\cdot\frac{\varphi'(s)}{\varphi(s)}\right)\,ds\right).
\end{equation}
\end{algor}

For the construction of the univalent log-harmonic mappings $f$ of the form $f(z)=zh(z)\overline{g(z)}$, the following steps
may be used:
\begin{enumerate}
\item Choose an arbitrary $\varphi\in \mathcal{S}^{*}$ and an arbitrary analytic function $\mu:\,\ID\to\ID$ with $\mu(0)=0$.
\item Establish the pair of equations
\begin{equation*}
\frac{zh(z)}{g(z)}=\varphi(z)\quad{\text{and}}\quad \frac{z g'(z)/g(z)}{1+z h'(z)/h(z)}=\mu(z).
\end{equation*}
\item Solving for $\left(\log g\right)'(z)$, and then integrating with normalization $g(0)=1$ yield \eqref{EPRLHG}.
\item The desired univalent log-harmonic mapping $f(z)=zh(z)\overline{g(z)}\in\mathcal{S}_{Lh}$ is then given by \eqref{EPRLHF}.
\end{enumerate}

Various choices of the conformal mapping $\varphi(z)$ and the dilatation $\mu(z)$ produce a number of univalent log-harmonic
mappings as demonstrated below.

\begin{example}\label{falpha}
Let $f_{\alpha}(z)=zh(z)\overline{g(z)}\in\mathcal{S}_{Lh}$ be a log-harmonic mapping in $\ID$ and satisfy the pair of equations
\begin{equation}\label{vpha}
\varphi_{\alpha}(z)=\frac{zh(z)}{g(z)}=\frac{z}{(1-z)^{2(1-\alpha)}}\quad{\text{and}}\quad \mu(z)=\frac{z g'(z)/g(z)}{1+z h'(z)/h(z)}=z,
\end{equation}
where $0\leq\alpha< 1$. The resulting nonlinear system
\begin{equation*}
\left\{\begin{aligned}
         z\left(\log h\right)'(z)-z\left(\log g\right)'(z)&=\frac{2(1-\alpha)z}{1-z}  \\
         \left(\log g\right)'(z)-z\left(\log h\right)'(z)&=1
                          \end{aligned}\right.
\end{equation*}
has the solution
\begin{equation*}
  \left(\log g\right)'(z)=\frac{1-z+2(1-\alpha)z}{(1-z)^2},
\end{equation*}
which by integration gives
\begin{equation*}
  \log g(z)=2(1-\alpha)\frac{z}{1-z}+(1-2\alpha)\log (1-z).
\end{equation*}
Thus, we have
\begin{equation}\label{eqGHA}
\left\{
\begin{aligned}
         g(z)&=(1-z)^{1-2\alpha}\exp\left(2(1-\alpha)\frac{z}{1-z}\right),  \\
         h(z)&=\frac{g(z)}{(1-z)^{2(1-\alpha)}}=\frac{1}{1-z} \exp\left(2(1-\alpha)\frac{z}{1-z}\right),
\end{aligned}
\right.
\end{equation}
and $f_{\alpha}(z)$ has the form
\begin{equation}\label{eqfa}
\begin{split}
  f_{\alpha}(z)&=\varphi_{a}(z)|g(z)|^2\\
  &=\frac{z}{(1-z)^{2(1-\alpha)}}
  \left|(1-z)^{1-2\alpha}\exp\left(2(1-\alpha)\frac{z}{1-z}\right)\right|^2.
\end{split}
\end{equation}

Note that $\varphi_{\alpha}(z)$ is starlike of order $\alpha$. By Theorem~\ref{thmLSS}, we know that the log-harmonic mapping
$f_{\alpha}(z)$ is also starlike of order $\alpha$ in $\ID$. Now we discuss two special cases for $\alpha=0$ and $\alpha=1/2$, respectively.

\begin{case}
In the case of $\alpha=0$, we have the Koebe function
\begin{equation*}
 \varphi_{0}(z)=\frac{z}{(1-z)^2}
\end{equation*}
and \eqref{eqGHA}, takes the form
\begin{equation}\label{eqGHAK}
\left\{
\begin{aligned}
g(z)&=(1-z)\exp\left(\frac{2z}{1-z}\right)=\exp\left(\sum_{n=1}^{\infty}\big(2+\frac{1}{n}\big)z^n\right),\\
h(z)&=\frac{1}{1-z} \exp\left(\frac{2z}{1-z}\right)=\exp\left(\sum_{n=1}^{\infty}\big(2-\frac{1}{n}\big)z^n\right).
\end{aligned}
\right.
\end{equation}
Finally \eqref{eqfa} takes the form
\begin{equation}\label{LKfeq}
f_{0}(z)=zh(z)\overline{g(z)}
=\frac{z}{(1-z)^2}|1-z|^2\exp\left(\RE\left(\frac{4z}{1-z}\right)\right),
\end{equation}
which is the well-known univalent \textbf{log-harmonic Koebe function}.
It is known that $f_{0}(z)$ maps $\ID$ onto the slit plane $f_{0}(\ID)=\mathbb{C}\backslash \{u+iv:\,u\leq -1/e^{2},v=0\}$.
This fact is known in \cite{Ali2016}, but the details were not given. For the sake of completeness, we include the details below.

Set $z=e^{i\theta},~0<\theta<2\pi$. A straightforward but tedious calculations show that
\begin{equation*}
\RE\left\{f_{0}(e^{i\theta})\right\}=-\frac{1}{e^2}\quad{\text{and}}\quad
\IM\left\{f_{0}(e^{i\theta})\right\}=0,
\end{equation*}
so that $f_{0}(z)=-\frac{1}{e^2}$ on the unit circle $|z|=1$ except at the point $z=1$. Now, consider the M\"{o}bius transformation
$$w=\frac{1+z}{1-z}=u+iv,
$$
which maps $\ID$ onto the right half-plane $\RE w=u>0$. Calculations show that
\begin{equation*}
\begin{split}
f_{0}(z)&=(w^2-1)\frac{1}{|w+1|^2}\exp\left(\RE\left(2(w-1)\right)\right)\\
&=\frac{u^2-v^2-1+i 2uv}{(u+1)^2+v^2}\exp(2(u-1)), \quad u>0.
\end{split}
\end{equation*}
Now we observe that:
\begin{enumerate}
\item Each point $z$ such that $|z|=1$ ($z\neq 1$) is carried onto a point $w$ on the imaginary axis so that $u=0$ and $f_{0}(z)=-1/e^2$.
\item If $uv=0$, we observe that the positive real axis
$$\{w=u+iv:\,u>0,v=0\}
$$
is mapped monotonically to the real interval $(-1/e^2,\infty)$.
\item Finally, each hyperbola $uv=c$, where $c\neq 0$ is a real constant, is carried univalently to the set
$$\left\{w_1=\frac{u^2-\left(\frac{c}{u}\right)^2-1}{(u+1)^2}\exp(2(u-1))+i 2c\exp(2(u-1)), \, u>0\right\},
$$
    which is the entire line $\{w_1=u_1+iv_1:\,-\infty<u_1<\infty\}$.
\end{enumerate}
Thus, the log-harmonic Koebe mapping $f_{0}(z)$ is univalent in $\ID$ and maps $\ID$ onto the entire plane minus the real
interval $(-\infty,-1/e^2]$.
\end{case}

\begin{figure}[!h]
  \centering
  \subfigure[The Koebe function $\frac{z}{(1-z)^2}$]
  {\begin{minipage}[b]{0.45\textwidth}
  \includegraphics[height=2.6in,width=2.6in,keepaspectratio]{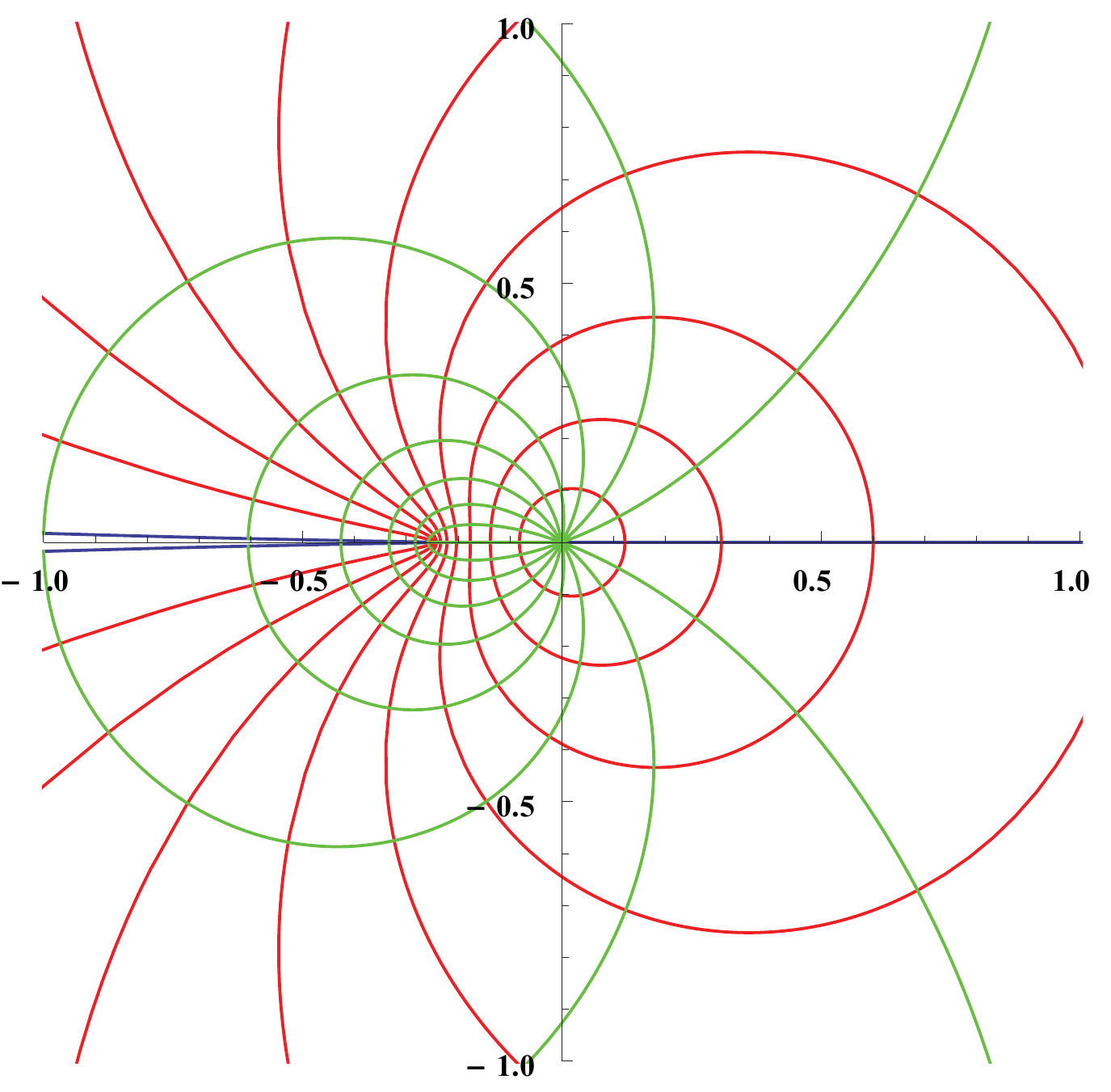}
\end{minipage}}
\subfigure[The log-harmonic Koebe mapping $f_0(z)$]
        {\begin{minipage}[b]{0.45\textwidth}
  \includegraphics[height=2.6in,width=2.6in,keepaspectratio]{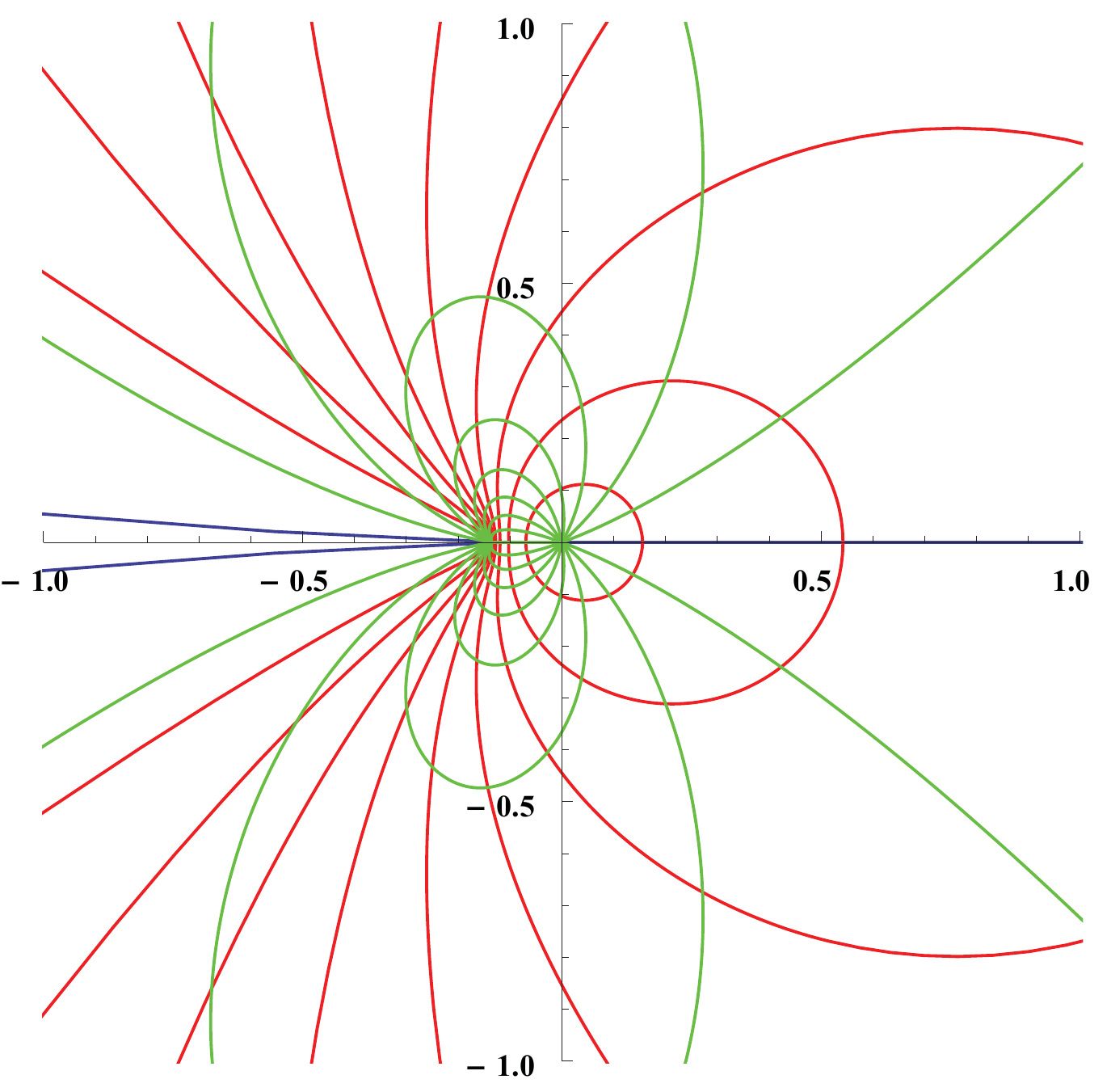}
\end{minipage}}
\caption{Images of $\ID$ under Koebe mapping and log-harmonic Koebe mapping }\label{fLKf}
\end{figure}

\begin{case}
For $\alpha=1/2$, by \eqref{vpha}, we have
\begin{equation*}
 \varphi_{\frac{1}{2}}(z)=\frac{z}{1-z}=l(z)
\end{equation*}
and from \eqref{eqGHA}, we obtain
\begin{equation*}
\left\{
\begin{aligned}
g(z)&=\exp\left(\frac{z}{1-z}\right)
         =\exp\left(\sum_{n=1}^{\infty} z^n\right),\\
h(z)&=\frac{1}{1-z} \exp\left(\frac{z}{1-z}\right)
         =\exp\left(\sum_{n=1}^{\infty}\big(1+\frac{1}{n}\big)z^n\right).
\end{aligned}
\right.
\end{equation*}
From \eqref{eqfa}, we obtain the univalent \textbf{log-harmonic right half-plane mapping}
\begin{equation}\label{LRfeq}
f_{\frac{1}{2}}(z)=\varphi(z)|g(z)|^2
=\frac{z}{1-z}\exp\left(\RE\left(\frac{2z}{1-z}\right)\right).
\end{equation}

We claim now that $f_{\frac{1}{2}}(\ID)=\left\{w:\, \RE w>-\frac{1}{2e}\right\}$. In order to prove this fact, we set
$z=e^{i\theta}~(0<\theta<2\pi)$ into~\eqref{LRfeq}. Then a straightforward calculation shows that
\begin{equation*}
\RE\left\{f_{\frac{1}{2}}(e^{i\theta})\right\}=-\frac{1}{2e}\quad{\text{and}}\quad
\IM\left\{f_{\frac{1}{2}}(e^{i\theta})\right\}=\frac{1}{2e}\cot\frac{\theta}{2}\in\mathbb{R},
\end{equation*}
so that $f_{\frac{1}{2}}(z)$ maps the unit circle $|z|=1$ ($z\neq 1$) onto the line $u=-\frac{1}{2e}$. With
$\zeta=\frac{z}{1-z}=a+i b$, where $a>-1/2,\,-\infty<b<\infty$ for $z\in\ID$, the log-harmonic mapping $f_{\frac{1}{2}}(z)$
takes the form
\begin{equation*}
f_{\frac{1}{2}}(z)=(a+i b)\exp(2a),\quad z=l^{-1}(\zeta)=\frac{\zeta}{1+\zeta}.
\end{equation*}
This shows that $f_{\frac{1}{2}}\circ l^{-1}$ maps each vertical line
$$\zeta=a_{0}+i b, \quad a_{0}>-1/2,\,-\infty<b<\infty,
$$
monotonically onto
$$\{w=u+i v:\, u=a_{0}\exp(2a_0)>-\frac{1}{2e},\, -\infty<v=b\exp(2a_0)<\infty\},
$$
where these lines correspond to circles in the unit disk. This shows that the mapping $w=f_{\frac{1}{2}}(z)$ sends $\ID$
univalently onto the right half-plane $\RE w>-\frac{1}{2e}$. The images under $f_{\frac{1}{2}}(z)$ of concentric circles and
radial segments are shown in Figure~\ref{LRf}. For a comparison we include the images of analytic right half-plane mapping and
the log-harmonic right half-plane mapping. See Figure~\ref{Rf}.
\end{case}
\begin{figure}[!h]
  \centering
  \subfigure[The right half-plane mapping $\frac{z}{1-z}$]
  {\begin{minipage}[b]{0.45\textwidth}
  \includegraphics[height=2.6in,width=2.6in,keepaspectratio]{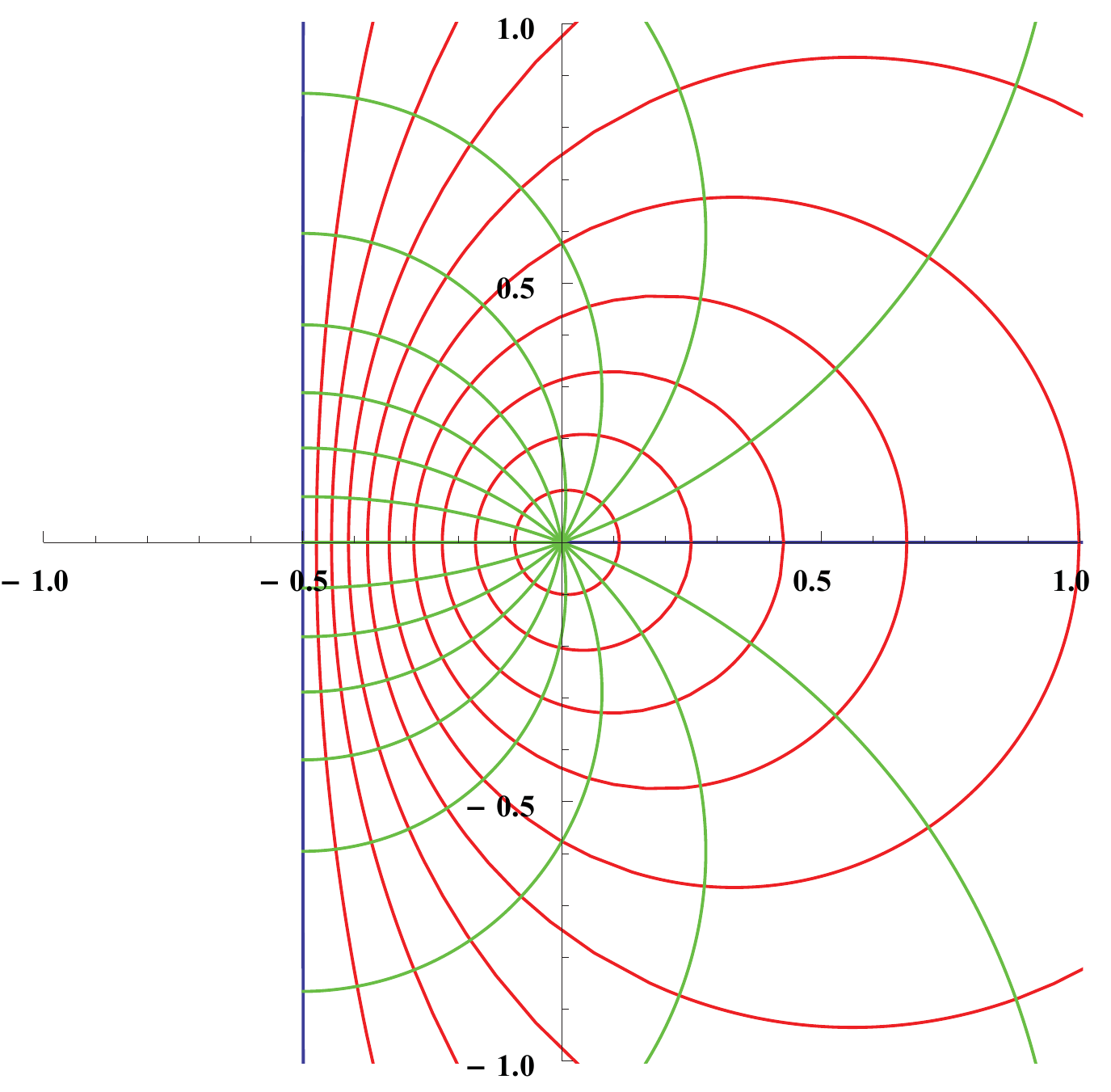}\label{Rf}
\end{minipage}}
\subfigure[The log-harmonic right half-plane mapping $f_{\frac{1}{2}}(z)$]
        {\begin{minipage}[b]{0.45\textwidth}
  \includegraphics[height=2.6in,width=2.6in,keepaspectratio]{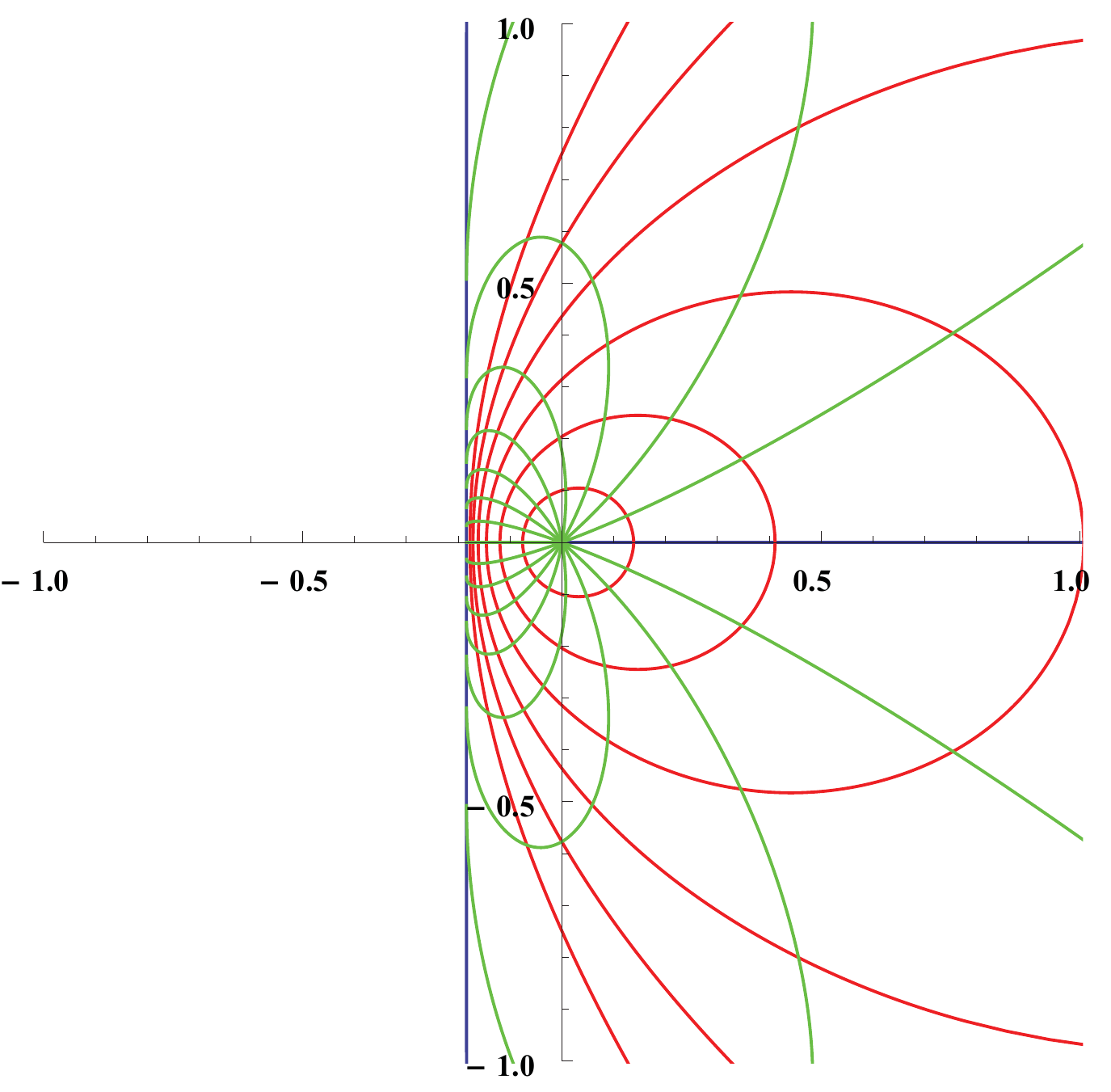}\label{LRf}
\end{minipage}}
\caption{Images of $\ID$ under right half-plane mapping and log-harmonic right half-plane mapping }\label{fLRf}
\end{figure}
\end{example}

\begin{example}({\bf Log-harmonic two-slits mapping}) We know that $s(z)=z/(1-z^2)$ maps $\ID$ onto
$\mathbb{C}\backslash\{u+iv:\,u=0,|v|\geq 1/2\}$. Let us now construct a log-harmonic two-slits mapping. Let
$LS(z)=zh(z)\overline{g(z)}\in\mathcal{S}_{Lh}$ with
\begin{equation*}
\varphi(z)=\frac{zh(z)}{g(z)}=s(z)=\frac{z}{1-z^2}\quad{\rm and}\quad
\mu(z)=\frac{zg'(z)/g(z)}{1+z h'(z)/h(z)}=z^2.
\end{equation*}
Thus, as before, we have
\begin{equation*}
z\left(\log h\right)'(z)-z\left(\log g\right)'(z)=\frac{2z^2}{1-z^2}\quad{\rm and}\quad
\left(\log g\right)'(z)-z^2 \left(\log h\right)'(z)=z.
\end{equation*}
Solving for the solution yields
$$\left(\log g\right)'(z)=\frac{z(1+z^2)}{(1-z^2)^2}.
$$
Integrating with the normalization $h(0)=g(0)=1$, we arrive at
\begin{equation*}
\left\{
\begin{aligned}
g(z)&=\sqrt{1-z^2}\exp\left(\frac{z^2}{1-z^2}\right)
=\exp\left(\sum_{n=1}^{\infty}\big(1-\frac{1}{2n}\big)z^{2n}\right),\\
h(z)&=\frac{g(z)}{1-z^2}=\frac{1}{\sqrt{1-z^2}}\exp\left(\frac{z^2}{1-z^2}\right)
=\exp\left(\sum_{n=1}^{\infty}\big(1+\frac{1}{2n}\big)z^{2n}\right),
\end{aligned}\right.
\end{equation*}
and thus,
\begin{equation}\label{LHTSM}
LS(z)=\varphi(z)|g(z)|^2
=\frac{z}{1-z^2}|1-z^2|\exp\left(\RE\left(\frac{2z^2}{1-z^2}\right)\right).
\end{equation}

By Theorem~\ref{thmLSS}, we know that $LS(z)$ is univalent and starlike in $\ID$. We now claim that $LS(z)$ maps $\ID$ onto
the two-slits plane $\mathbb{C}\backslash\{u+iv:\, |v|\geq 1/e\}$. The images of $\ID$ under $s(z)=z/(1-z^2)$ and log-harmonic
two-slits mapping $LS(z)$ are shown in Figures~\ref{S2f} and~\ref{LS2f}, respectively.

In order to prove that $LS(\ID)=\mathbb{C}\backslash\{u+iv:\, |v|\geq 1/e,u=0\}$, set $z=e^{i\theta}~(\theta\in(0,\pi)\cup(\pi,2\pi))$
into~\eqref{LHTSM}. We have
\begin{equation*}
LS(e^{i\theta})=i \frac{\sin\theta}{|\sin\theta|}e^{-1}
=\left\{ \begin{aligned}
         i/e&\quad{\rm for}\quad 0<\theta<\pi,  \\
        -i/e&\quad{\rm for}\quad \pi<\theta<2\pi,
                          \end{aligned} \right.
\end{equation*}
which shows that
$LS(z)= \pm i/e$ on the unit circle except at the points $z=\pm 1$. The argument similar to the analysis of Example~\ref{falpha}
gives the desired claim and we omit the details.
\end{example}
\begin{figure}[!h]
  \centering
  \subfigure[The two-slits function $\frac{z}{1-z^2}$]
  {\begin{minipage}[b]{0.45\textwidth}
  \includegraphics[height=2.6in,width=2.6in,keepaspectratio]{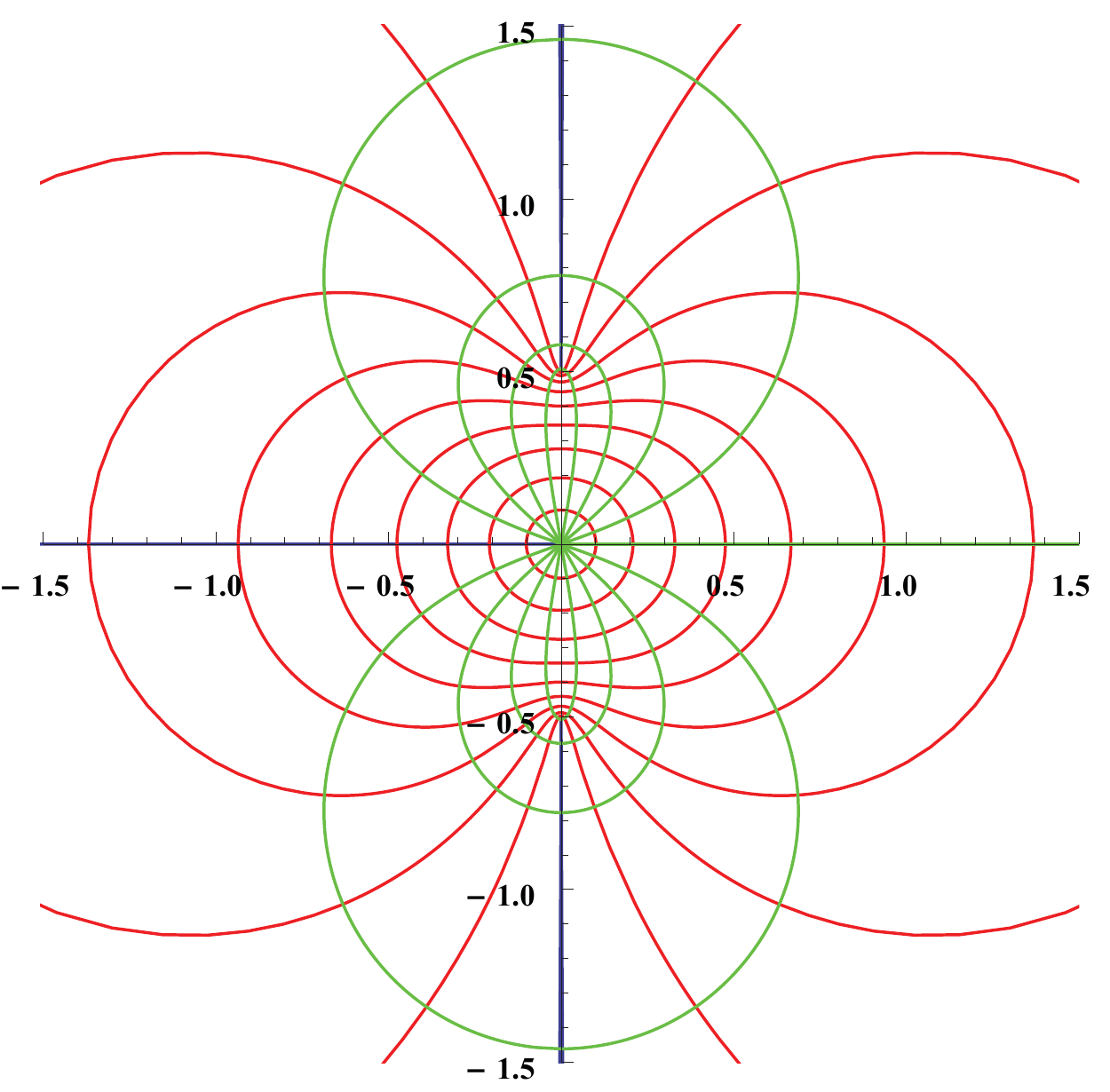}\label{S2f}
\end{minipage}}
\subfigure[The log-harmonic two-slits function $LS(z)$]
        {\begin{minipage}[b]{0.45\textwidth}
  \includegraphics[height=2.6in,width=2.6in,keepaspectratio]{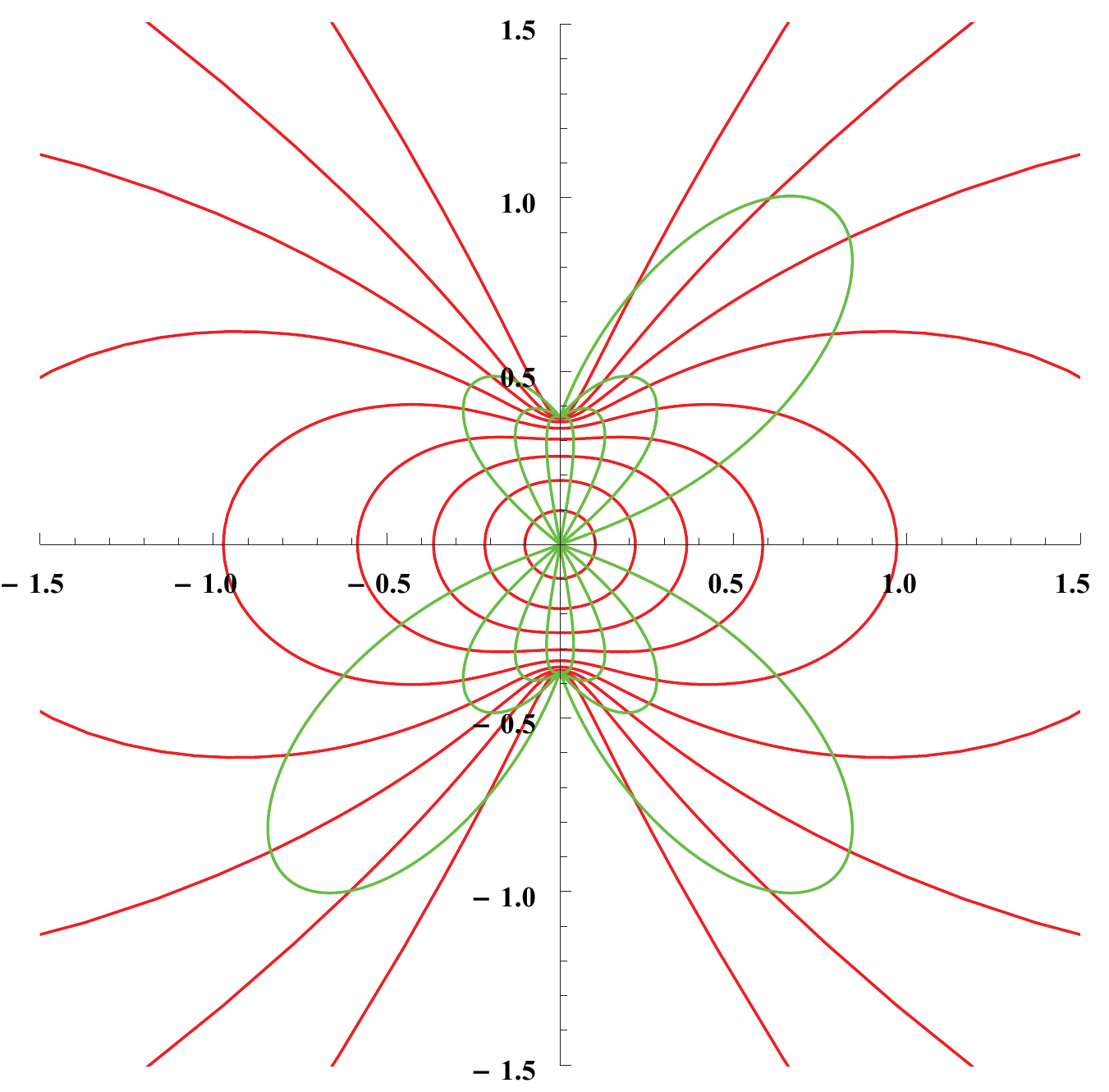}\label{LS2f}
\end{minipage}}
\caption{Images of $\ID$ under $s(z)$ and $LS(z)$}\label{fLSf}
\end{figure}

\begin{remark}
The above three univalent log-harmonic mappings play the role of extremal functions for many extremal problems
over the subclasses of $\mathcal{S}_{Lh}$.
\end{remark}
\section{Coefficients estimate for log-harmonic starlike mappings}\label{Sec3}

Let $s_1(z)=\sum_{n=0}^{\infty}a_n z^n$ and $s_2(z)=\sum_{n=0}^{\infty}b_n z^n$ be analytic functions in $\ID$. We say that $s_1(z)$
is \emph{subordinate} to $s_2(z)$ (written by $s_1(z)\prec s_2(z)$) if
$$s_1(z)=s_2(\omega(z))
$$
for some analytic function $\omega:\,\ID\rightarrow \ID$ with $\omega(0)=0$. Then, by the Schwarz lemma, $|\omega (z)|\leq |z|$ and
$|\omega '(0)|\leq 1$ so that  $|s_1'(0)|\leq |s_2'(0)|$. For additional details on subordination classes,
see for example~\cite[Chapter 6]{Duren1983} or~\cite[p. 35]{Pommerenke1975}.

\begin{lemma}\label{SCSB}{\rm (\cite[Theorem 6.4]{Duren1983})}
Let $s_1(z)\prec s_2(z)$, where $s_1(0)=s'_1(0)-1=0$ and $s_2(0)=s'_2(0)-1=0$. Then we have the following:
\begin{enumerate}
\item if $s_2\in \mathcal{C}$, then $|a_n|\leq 1$ for $n=2,3,\cdots$;
\item if $s_2\in \mathcal{S^{*}}$, then $|a_n|\leq n$ for $n=2,3,\cdots$.
\end{enumerate}
\end{lemma}

\begin{theorem}\label{thmSTLh}
Let $f(z)=zh(z)\overline{g(z)}$ belong to $\mathcal{S}_{Lh}^{*}(\alpha)~(0\leq\alpha<1)$, where $h(z)$ and $g(z)$
are given by~\eqref{LHBJC}. Then
$$\left|a_{n}-b_{n}\right|\leq \frac{2(1-\alpha)}{n}\quad{\rm for}\,\, n\geq 1.
$$
Equality holds if $f(z)=f_{\alpha}(z)$ or one of its rotation, where where $f_{\alpha}(z)$ is given by~\eqref{eqfa}.
\end{theorem}
\begin{proof}
Let $f(z)=zh(z)\overline{g(z)}$ be an element of $\mathcal{S}_{Lh}^{*}(\alpha)$. Then we have
\begin{equation*}
\alpha<\RE\left(\frac{zf_{z}(z)-\overline{z}f_{\overline{z}}(z)}{f(z)}\right)
=\RE\left(1+\frac{zh'(z)}{h(z)}-\frac{zg'(z)}{g(z)}\right), \quad z\in\ID.
\end{equation*}
By ~\eqref{LHBJC} and Theorem ~\ref{thmLSS}, we obtain
\begin{equation*}
1+\frac{zh'(z)}{h(z)}-\frac{zg'(z)}{g(z)}\prec \frac{1+(1-2\alpha)z}{1-z}, \quad z\in\ID,
\end{equation*}
which is equivalent to
$$\frac{1}{2(1-\alpha)}\left(\frac{zh'(z)}{h(z)}-\frac{zg'(z)}{g(z)}\right)
=\sum_{n=1}^{\infty}\frac{n(a_n-b_n)}{2(1-\alpha)}z^n\prec \frac{z}{1-z}, \quad z\in\ID.
$$
Also, since $z/(1-z)$ is convex in $\ID$, by Lemma~\ref{SCSB}, we get
$$\frac{n\left|a_n-b_n\right|}{2(1-\alpha)}\leq 1\quad{\rm for}\quad n\geq 1,
$$
and the desired coefficient inequality follows.

Finally, it is evident that the equalities are attained by a suitable rotation of $h(z)$ and $g(z)$ given by \eqref{eqfa}.
The proof is complete.
\end{proof}

For $\alpha =0$, one has the following.

\begin{corollary}\label{corSTLh}
Let $f(z)=zh(z)\overline{g(z)}$ be an element of  $\mathcal{S}_{Lh}^{*}$, where $h(z)$ and $g(z)$ are given by~\eqref{LHBJC}. Then
$$\left|a_{n}-b_{n}\right|\leq \frac{2}{n} ~\mbox{ for $n\geq 1$}.
$$
Equality holds if $f(z)=f_{0}(z)$ or one of its rotation, where $f_{0}(z)$ is given by~\eqref{LKfeq}.
\end{corollary}

\begin{theorem}\label{thmSSTLh}
Let $f(z)=zh(z)\overline{g(z)}$ be a log-harmonic mapping in $\ID$, where $h(z)$ and $g(z)$ are given by~\eqref{LHBJC}, and satisfy the condition
\begin{equation}\label{CDSumS}
\sum_{n=1}^{\infty}n\left|a_n-b_n\right|\leq 1-\alpha
\end{equation}
for some $\alpha\in[0,1)$. Then $f\in\mathcal{S}_{Lh}^{*}(\alpha)$.
\end{theorem}
\begin{proof}
By using the series representation of $h(z)$ and $g(z)$ given by~\eqref{LHBJC}, we obtain
\begin{eqnarray*}
\RE\left(\frac{zf_{z}(z)-\overline{z}f_{\overline{z}}(z)}{f(z)}\right)
&&=\RE\left(1+\sum_{n=1}^{\infty}n(a_n-b_n)z^n\right)\\
&&\geq 1-\left|\sum_{n=1}^{\infty}n(a_n-b_n)z^n\right|\\
&&> 1-\sum_{n=1}^{\infty}n\left|a_n-b_n\right|\geq \alpha.
\end{eqnarray*}
By ~\eqref{CDSumS}, the desired conclusion follows.
\end{proof}

In particular, $\alpha=0$ in Theorem \ref{thmSSTLh} provides a sufficient coefficient condition for the log-harmonic
mappings of the form $f(z)=zh(z)\overline{g(z)}$ to be starlike in $\ID$.

\section{Growth and Distortion Theorem}\label{Sec4}
In this section, we introduce the subclass $\mathcal{C}_{Lh}$ of $\mathcal{S}_{Lh}$, which yields sharp growth and distortion
estimates, where $\mathcal{C}_{Lh}$ is defined by
\begin{equation*}
\mathcal{C}_{Lh}=\left\{f\in\mathcal{S}_{Lh}:\, f(z)=zh(z)\overline{g(z)},~ \frac{zh(z)}{g(z)} =\frac{z}{1-z}
\right\}.
\end{equation*}

The following two theorems are the growth theorem and distortion theorem for the class $\mathcal{C}_{Lh}$.

\begin{theorem}\label{GCthm}
Let $f(z)=zh(z)\overline{g(z)}\in\mathcal{C}_{Lh}$. Then for $z\in\ID$ we have
\begin{enumerate}
\item $\ds |h(z)|\leq\frac{1}{1-|z|}
    \exp\left(\frac{|z|}{1-|z|}\right)$;
\item $\ds |g(z)|\leq \exp\left(\frac{|z|}{1-|z|}\right)$;
\item $\ds |f(z)|\leq \frac{|z|}{1-|z|}\exp\left(\frac{2|z|}{1-|z|}\right)$.
\end{enumerate}
The equalities occur if and only if $f(z)$ is  of the form $\overline{\eta} f_{\frac{1}{2}}(\eta z),\, |\eta|=1$,
where $f_{\frac{1}{2}}(z)$ is given by~\eqref{LRfeq}.
\end{theorem}
\begin{proof}
Let $f(z)=zh(z)\overline{g(z)}\in\mathcal{C}_{Lh}$. It follows from ~\eqref{EPRLHG} that
\begin{equation}\label{eqG}
g(z)=\exp\left(\int_{0}^{z}\left(\frac{\mu(s)}{1-\mu(s)}
\cdot\frac{1}{s(1-s)}\right)\,ds\right),
\end{equation}
where $\mu\in \mathcal{B}$ such that $\mu(0)=0$, and by the definition of $\mathcal{C}_{Lh}$,
\begin{equation}\label{eqHF}
h(z)=\frac{1}{1-z}g(z)\quad{\rm and}\quad f(z)=\frac{z}{1-z}|g(z)|^2.
\end{equation}
Because $\mu\in \mathcal{B}$ with $\mu(0)=0$, for $|z|=r$, we have
$$\left|\frac{\mu(z)}{z(1-\mu(z))}\right|\leq \frac{1}{1-r}\quad{\rm and}\quad \left|\frac{1}{1-z}\right|\leq\frac{1}{1-r},
$$
which imply that
$$|g(z)|\leq\exp\left(\int_{0}^{r}\frac{1}{(1-t)^2}\,dt\right)
=\exp\left(\frac{r}{1-r}\right), ~ |h(z)|\leq\frac{1}{1-r}\exp\left(\int_{0}^{r}\frac{1}{(1-t)^2}\,dt\right),
$$
and thus,
\begin{equation*}
|f(z)|=\left|\frac{z}{1-z}\right||g(z)|^2\leq\frac{r}{1-r}\exp\left(\frac{2r}{1-r}\right).
\end{equation*}
Equalities occur if and only if $\mu(z)=\eta z,\,|\eta|=1$ which leads to $f(z)=\overline{\eta} f_{\frac{1}{2}}(\eta z)$.
\end{proof}

\begin{theorem}\label{DTthm}
Let $f(z)=zh(z)\overline{g(z)}\in\mathcal{C}_{Lh}$. Then for $z\in\ID$ we have
\begin{enumerate}
\item\label{item:1}  $\ds |f_{z}(z)|\leq \frac{1}{(1-|z|)^3}\exp\left(\frac{2|z|}{1-|z|}\right)$;
\item\label{item:2}  $\ds |f_{\overline{z}}(z)|\leq \frac{|z|}{(1-|z|)^3}\exp\left(\frac{2|z|}{1-|z|}\right)$;
\item\label{item:3}  $\ds |Df(z)|\leq \frac{|z|(1+|z|)}{(1-|z|)^3}\exp\left(\frac{2|z|}{1-|z|}\right)$.
\end{enumerate}
The equalities occur if and only if $f(z)$ is of the form $\overline{\eta} f_{\frac{1}{2}}(\eta z),\, |\eta|=1$,
where $f_{\frac{1}{2}}(z)$ is given by~\eqref{LRfeq}.
\end{theorem}
\begin{proof}
Let $f(z)=zh(z)\overline{g(z)}\in\mathcal{C}_{Lh}$. In view of~\eqref{eqG} and ~\eqref{eqHF}, we find that
\begin{eqnarray*}
f_{z}(z)=(h(z)+zh'(z))\overline{g(z)}= \left (1+\frac{zh'(z)}{h(z)}\right )\frac{h(z)}{g(z)}|g(z)|^2.
\end{eqnarray*}
Now the relations
\begin{equation*}
\frac{h(z)}{g(z)}=\frac{1}{1-z}\quad{\rm and}\quad 1-\mu(z)=\frac{1+zh'(z)/h(z)-zg'(z)/g(z)}{1+zh'(z)/h(z)}
\end{equation*}
give
\begin{equation*}
1+\frac{zh'(z)}{h(z)}-\frac{zg'(z)}{g(z)}=\frac{1}{1-z}
\end{equation*}
and
\begin{equation*}
1+\frac{zh'(z)}{h(z)}=\frac{1}{(1-\mu(z))(1-z)},
\end{equation*}
so that $f_{z}(z)$ takes the form
\begin{equation*}
f_{z}(z)=\frac{1}{(1-\mu(z))(1-z)^2}|g(z)|^2,
\end{equation*}
where $\mu\in \mathcal{B}$ with $\mu(0)=0$. Similarly, we see that
\begin{eqnarray*}
f_{\overline{z}}(z)=zh(z)\overline{g'(z)}
=\overline{\left(\frac{\mu(z)/z}{(1-\mu(z))(1-z)}\right)}\cdot\frac{z}{1-z}|g(z)|^2.
\end{eqnarray*}
For $|z|=r$, we have
$$\left|\frac{1}{1-\mu(z)}\right|\leq \frac{1}{1-r}\quad{\rm and}\quad
\left|\frac{\mu(z)/z}{1-\mu(z)}\right|\leq \frac{1}{1-r}
$$
so that
\begin{eqnarray*}
|f_{z}(z)|\leq\frac{1}{(1-r)^3}\exp\left(\frac{2r}{1-r}\right)\quad{\rm and}\quad
|f_{\overline{z}}(z)|\leq\frac{r}{(1-r)^3}\exp\left(\frac{2r}{1-r}\right),
\end{eqnarray*}
which prove \eqref{item:1} and \eqref{item:2}. Using these two inequalities,  we obtain that
\begin{equation*}
|Df(z)|=|zf_{z}(z)-\overline{z}f_{\overline{z}}(z)|
\leq|z|\left(|f_{z}(z)|+|f_{\overline{z}}(z)|\right)
\leq \frac{r(1+r)}{(1-r)^3}\exp\left(\frac{2r}{1-r}\right),
\end{equation*}
which proves \eqref{item:3}.

Finally, equalities occur if and only if $\mu(z)=\eta z,\,|\eta|=1$ which leads to $f(z)=\overline{\eta} f_{\frac{1}{2}}(\eta z)$.
\end{proof}

\section{Open problems}\label{Sec5}
The function $f_{0}$  given by \eqref{LKfeq} plays the role of the Koebe mapping in the set of log-harmonic mappings
(see also~\cite{Abdulhadi1988}). As an analog of analytic and harmonic Bieberbach conjectures, it is natural to propose the following:

\begin{conj}\label{LHBJ}({\bf Log-harmonic Coefficient Conjecture})
Let $f(z)=zh(z)\overline{g(z)}\in\mathcal{S}_{Lh}$, where  $h$ and $g$ are given by~\eqref{LHBJC}. Then for all $n\geq 1$,
\begin{enumerate}
\item\label{item:a} $\ds |a_{n}|\leq 2+ \frac{1}{n}$;
\item\label{item:b} $\ds |b_{n}|\leq 2- \frac{1}{n}$;
\item\label{item:c} $\ds \left|a_{n}-b_{n}\right|\leq \frac{2}{n}$.
\end{enumerate}
\end{conj}

We remark that if \eqref{item:b} and \eqref{item:c} hold in Conjecture \ref{LHBJ}, then, by \eqref{item:c}, we can obtain
\begin{equation*}
|a_n|\leq\frac{2}{n}+|b_n|\leq\frac{2}{n}+2-\frac{1}{n}=2+\frac{1}{n}.
\end{equation*}
Thus, it suffices to prove \eqref{item:b} and \eqref{item:c}.
It is worth pointing out that Conjecture \ref{LHBJ} is true for starlike log-harmonic mappings, see
\cite[Theorem 3.3]{Abdulhadi1989} and Corollary~\ref{corSTLh}.

\begin{conj}\label{LHCTJ}({\bf  Log-harmonic $1/e^2$-Covering Conjecture})
For $f\in \mathcal{S}_{Lh}$, we conjecture that $f(\ID)$ covers the disk $\{w\in\mathbb{C}:\,|w|<1/e^2\}$.
\end{conj}

For log-harmonic Koebe mapping $f_{0}(z)$ defined by \eqref{LKfeq}, the constant $1/e^2$ (see Example~\ref{falpha}) cannot be improved.
In \cite{Abdulhadi1988}, it was shown that the range of $f\in \mathcal{S}_{Lh}$ covers the disk $\{w\in \mathbb{C}:\, |w|<1/16\}$.




\begin{center}{\sc Acknowledgements}
\end{center}

\vskip.05in
The research of the first author was supported by the Foundation of Guilin University of Technology under Grant No.GUTQDJJ2018080, the Natural Science Foundation of Guangxi (Grant No. 2018GXNSFAA050005) and Hunan Provincial Key Laboratory of Mathematical Modeling and Analysis in Engineering (Changsha University of Science \& Technology).
The work of the second author is supported by Mathematical Research Impact Centric Support of  Department of Science and Technology (DST),
India~(MTR/2017/000367).




\end{document}